\documentclass[a4paper, 12pts]{amsart}
\usepackage{amssymb}
\usepackage{stackengine}
\usepackage{geometry}
 \usepackage[colorlinks=true]{hyperref}
\usepackage{latexsym}
\usepackage{amsmath}
\usepackage{amsthm}
\usepackage{amsfonts}
\usepackage{amssymb}
\usepackage{mathtools}
\usepackage{esint}
\usepackage{bbm}
\usepackage{verbatim}
\usepackage{framed}
\usepackage{hyperref}

 \allowdisplaybreaks
\usepackage{pstricks}

\newtheorem{Th}{Theorem}
\newtheorem{Lm}{Lemma}

\title{Index stability for two dimensional bubble tree convergence with $W^{1,2}$ quantization and applications}

\title{Optimal weighted Wente's inequality}%

\author{Matilde Gianocca}
\address{Department of Mathematics, ETH Zentrum,
CH-8093 Z\"urich, Switzerland.}
\email{matilde.gianocca@math.ethz.ch}

\usepackage{lipsum}
\usepackage{mathtools}  
\mathtoolsset{showonlyrefs}  
\begin{document}

\begin{abstract}
  In this short note we show three types of weighted Wente inequalities: $W^{1,2}_0(B_1)$ solutions of $\Delta\varphi=\nabla a\cdot\nabla^\perp b$ for $a,b\in W^{1,2}(B_1)$ satisfy \[||\,|x|^\alpha\varphi||_{L^{\infty}(B_1)}^2+\int_{B_1}|x|^{2\alpha}|\nabla\varphi|^2\,dx\leq C\int_{B_1}|x|^{2\alpha}|\nabla b|^2\,dx\int_{B_1}|\nabla a|^{2}\,dx\,,\]
for any $\alpha\in (0,1)$. For the case $\alpha=1$ we show
\[||\,|x|\varphi||_{L^{\infty}(B_1)}^2\leq C\int_{B_1}|x|^{2}|\nabla b|^2\,dx\int_{B_1}|\nabla a|^{2}\,dx\,,\]
\begin{equation}\label{maineq3}
    \int_{B_1}|x|^{2}|\nabla\varphi|^2\,dx\leq C\int_{B_1}|x|^{2}|\log|x||\,|\nabla b|^2\,dx\int_{B_1}|\nabla a|^{2}\,dx\,.
\end{equation}
The last inequality is an improvement of a previous result obtained by the author in collaboration with F. Da Lio and T. Rivière in \cite{DGR} using different methods.
The case $\alpha=0$ reduces to the original Wente estimate and is therefore already established. We prove in Theorem 2 that \eqref{maineq3} is optimal.
\end{abstract}

\maketitle
 \thispagestyle{empty}
\section*{Introduction}
\noindent Let us consider weak solutions of the Poisson equation \[\Delta f = g\] on the unit ball $B_1\subset\mathbb{R}^2$ with Dirichlet boundary conditions $f=0$ on $\partial B_1$. Under the assumption that $f\in L^1(B_1)$ one cannot expect $f\in W^{1,2}(B_1)$ nor $f\in L^{\infty}(B_1)$ in general. The optimal a-priori estimates are
\[f\in \bigcap\limits_{1\leq q <\infty}L^q(B_1)\quad\text{and}\quad \nabla f\in\bigcap\limits_{1\leq p<2}L^p(B_1)\,,\]
as well as
\[f\in \text{BMO}(B_1)\quad \text{and}\quad \nabla f\in L^{2,\infty}(B_1)\,.\]
We recall that the weak $L^2$ space $L^{2,\infty}(B_1)\supsetneq L^2(B_1)$, is given by
\[L^{2,\infty}(B_1)=\{f\,\text{measurable, s.t. }\sup\limits_{\lambda\geq 0}\lambda|\{x\in B_1\,\vert\,|f(x)|\geq\lambda\}|^{1/2}<\infty\}\,.\]
In \cite{Wen} H. Wente discovered improved regularity properties for functions whose Laplacian is a Jacobian determinant, i.e. for solutions of 
\begin{equation}\label{wenteequation}
    \begin{cases}
        \Delta \varphi =\nabla a\cdot\nabla^\perp b\quad&\text{on }B_1\\[.5em]
        \,\,\,\,\varphi=0 \quad&\text{on }\partial B_1\,.
    \end{cases}
\end{equation}
Assuming that $a,b\in W^{1,2}(B_1)$ where $B_1\subset\mathbb{R}^2$ denotes the unit disk, one has that the Jacobian $\nabla^\perp a\cdot\nabla b$ is in $L^1$, and as explained above having $\Delta\varphi\in L^1(B_1)$ is not sufficient to obtain $L^2$ integrability of $\nabla\varphi$ a priori. Wente's results show that thanks to the additional structure of the Jacobian, $\varphi$ is indeed in $W^{1,2}\cap L^\infty$ and the following estimate holds:
\begin{equation}\label{originalresultwente}
    ||\varphi||_{L^\infty(B_1)}+||\nabla \varphi||_{L^2(B_1)}\leq C||\nabla a||_{L^2(B_1)}||\nabla b||_{L^2(B_1)}\,.
\end{equation}
In \cite{BG} F. Bethuel and J. M. Ghidaglia generalized the estimate \eqref{originalresultwente} to arbitrary open domains $\Omega\subset\mathbb{R}^2$, not necessarily conformal to the open disk (in this case the estimate would be immediate since the equation and the norms involved are all invariant under conformal transformations).\\
In \cite{Ge} Y. Ge found the optimal constants for the inequality
\begin{equation*}
    ||\nabla\varphi||_{L^2(\mathcal{M})}\leq C_\mathcal{M}||\nabla a||_{L^2(\mathcal{M})}||\nabla b||_{L^2(\mathcal{M})}
\end{equation*}
for any smooth two dimensional Riemann Surface\footnote{both the cases $\partial \mathcal{M}=\varnothing$ and $\partial \mathcal{M}\neq\varnothing$ have been covered, in the first case the Dirichlet boundary condition is replaced by $\varphi(x_0)=0$ for some $x_0\in \mathcal{M}$\,.} $\mathcal{M}$, showing a connection between Wente's inequality and the isoperimetric inequality in $\mathbb{R}^3$. The optimal constant for the supremum norm $||\varphi||_{L^\infty(\mathcal{M})}$ was studied by S. Baraket \cite{Bar}, P. Topping \cite{Top} and H. Wente \cite{Wen2} and is now known to be $\frac{1}{2\pi}$ for every Riemann Surface $\mathcal{M}$, that is
\[||\varphi||_{L^\infty(\mathcal{M})}\leq\frac{1}{2\pi}||\nabla a||_{L^2(\mathcal{M})}||\nabla b||_{L^2(\mathcal{M})}\]\\
It has been shown by F. Da Lio and F. Palmurella in \cite{DLP} and independently by J. Hirsch in \cite{Hir}, that uniform estimates of the type \eqref{originalresultwente} for $||\nabla\varphi||_{L^2}$ and $||\varphi||_{L^\infty}$ of solutions of Jacobian equations with Neumann boundary conditions
\begin{equation*}
    \begin{cases}
        \,\Delta \varphi =\nabla a\cdot\nabla^\perp b\quad&\text{on }B_1\\[.5em]
        \partial_\nu\varphi=0 \quad&\text{on }\partial B_1\,.
    \end{cases}
\end{equation*}
do not exist in general.\\
Lorentz spaces turned out to play an important role in the study of compensation phenomena. We recall the definition of the $(p,q)$-quasinorms for $0<p,q<\infty$,
\[||f||_{L^{p,q}(B_1)}=p^{1/q}\bigg(\int_0^\infty t^q|\{x\,:\,|f(x)|\geq t\}|^{q/p}\,\frac{dt}{t}\bigg)^{1/q}\,.\]
The inequalities discovered by Wente were later studied and generalized by R. Coifmann, P.-L. Lions, Y. Meyer and S. Semmes in \cite{CLMS}, who showed that solutions of \eqref{wenteequation} satisfy
\begin{equation}\label{CLMSeq}
   ||\nabla\varphi||_{L^{2,1}(B_1)} +||\nabla^2\varphi||_{L^1(B_1)}\leq C||\nabla a||_{L^2(B_1)}||\nabla b||_{L^2(B_1)}.
\end{equation}
A related result due to F. Bethuel \cite{Bet} is
\begin{equation*}
    ||\nabla\varphi||_{L^{2}(B_1)}\leq C||\nabla a||_{L^{2,\infty}(B_1)}||\nabla b||_{L^2(B_1)}\,,
\end{equation*}
where the Jacobian $\nabla a\cdot \nabla^\perp b=\text{div}(a\nabla^\perp b)$ should be understood in the distributional sense for $\nabla a\in L^{2,\infty}$. \\[.5em]
Motivated by geometric quantization problems, P. Laurain and T. Rivière studied the Wente inequality on degenerating annuli,
\begin{equation*}
    \begin{cases}
        \Delta \varphi =\nabla a\cdot\nabla^\perp b\quad&\text{on }B_1\setminus B_\epsilon\\[.5em]
        \,\,\,\,\varphi=0 \quad&\text{on }\partial (B_1\setminus B_\epsilon)\,
    \end{cases}
\end{equation*}
and obtained\footnote{the original result is stated for $||\nabla\varphi||_{L^{2,1}(B_1\setminus B_{2\epsilon})}$, however the original proof together with a scaling argument gives the statement as it is presented here.} (see \cite{LR})
\[||\nabla\varphi||_{L^{2,1}(B_1\setminus B_\epsilon)}\leq C||\nabla a||_{L^2(B_1\setminus B_\epsilon)}||\nabla b||_{L^2(B_1\setminus B_\epsilon)}\]
for a constant $C$ independent of $\epsilon$.
\\[.5em]Wente-type estimates for solutions of \eqref{wenteequation} of the type
\begin{equation}
    ||\nabla \varphi||_{L^{p,q}(B_1)}\leq C_{p,q}||\nabla a||_{L^{2,\infty}(B_1)}||\nabla b||_{L^{p,q}(B_1)}
\end{equation}
for $p\in (1,\infty)$ and $q\in[1,\infty]$ have been shown by Y. Ge in \cite{Ge2} and with different methods by Y. Bernard and T. Rivière in \cite{BR}.
An exposition of Wente's inequality and related problems can also be found in Chapter 3 of F. Hélein's book \cite{Hel}.\\[0.5em]
In a recent work \cite{DGR}, the author in collaboration with F. Da Lio and T. Rivière  proved the following version of the weighted Wente inequality: for solutions $\varphi$ of equation \eqref{wenteequation}, it holds that
\begin{equation}\label{notoptimalweightedwente}
    \int_{B_1}|x|^2|\nabla\varphi|^2\,dx\leq C\int_{B_1}f(|x|)|\nabla b|^2\,dx\int_{B_1}|\nabla a|^2\,dx
\end{equation}
with the weight
\begin{equation}
    f(|x|)=|x|^2\log^2(1+|x|^{-1})\log(1+\log (|x|^{-1}))\,.
\end{equation}
This version of the weighted Wente inequality was fundamental in the study of Morse index stability for conformally invariant Langrangians presented in \cite{DGR}.\\[.5em]
In this note we study on the one hand the $L^\infty$ version of the weighted Wente inequality for weights $|x|^\alpha$, $\alpha\leq 1$. On the other hand we prove optimal $W^{1,2}$ weighted Wente estimates, which show that the weight $f$ obtained in \cite{DGR} was not sharp. More precisely,
\begin{framed}
\begin{Th}\label{wentetheorem}
    Let $\varphi\in W^{1,2}_0(B_1)$ be the unique solution of
    \begin{equation*}
    \begin{cases}
        \Delta \varphi =\nabla a\cdot\nabla^\perp b\quad&\text{on }B_1\\[.5em]
        \,\,\,\,\varphi=0 \quad&\text{on }\partial B_1\,,
    \end{cases}
    \end{equation*}
    for $a,b\in W^{1,2}(B_1)$, where $B_1\subset\mathbb{R}^2$ is the two-dimensional open disk. Then for any $0<\alpha<1$ there exists a constant $C_\alpha>0$, s.t.
    \begin{equation*}   ||\,|x|^\alpha\varphi||_{L^\infty(B_1)}^2+\int_{B_1}|x|^{2\alpha}|\nabla\varphi|^2\,dx\leq C_\alpha\int_{B_1}|x|^{2\alpha}|\nabla b|^2\,dx\int_{B_1}|\nabla a|^2\,dx\,.
    \end{equation*}
    Moreover there exists a universal constant $C>0$ such that
      \begin{equation*}   ||\,|x|\varphi||_{L^\infty(B_1)}^2\leq C\int_{B_1}|x|^{2}|\nabla b|^2\,dx\int_{B_1}|\nabla a|^2\,dx\,.
    \end{equation*}
        \begin{equation*}   \int_{B_1}|x|^{2}|\nabla\varphi|^2\,dx\leq C\int_{B_1}|x|^{2}|\log|x||\,|\nabla b|^2\,dx\int_{B_1}|\nabla a|^2\,dx\,.
    \end{equation*}
\end{Th}
\end{framed}
\begin{framed}
    \begin{Th}
        Let $\varphi\in W^{1,2}_0(B_1)$ be the unique solution of
    \begin{equation*}
    \begin{cases}
        \Delta \varphi =\nabla a\cdot\nabla^\perp b\quad&\text{on }B_1\\[.5em]
        \,\,\,\,\varphi=0 \quad&\text{on }\partial B_1\,,
    \end{cases}
    \end{equation*}
    for $a,b\in W^{1,2}(B_1)$, where $B_1\subset\mathbb{R}^2$ is the two-dimensional open disk. Denote by $\mathcal{A}$ the set of all such $\varphi$, for varying $a,b\in W^{1,2}(B_1)$. Then,
    \begin{equation*}
        \forall \beta<1:\quad\sup\limits_{\varphi\in\mathcal{A}}\frac{||\,|x|\nabla\varphi||_2}{||\nabla a||_2||\,|x|\,|\log|x||^{\beta/2}\nabla b||_2}=\infty\,.
    \end{equation*}
    \end{Th}
\end{framed}

\noindent No other results of weighted Wented estimates of this form are known to the author.\\
Wente inequalities on weighted Sobolev spaces have been shown by S. Baraket and L. Chaabane in \cite{BC}. Despite the similar name, these results are fundamentally different from the ones mentioned above, since in their result, given a weight $w$, the integrability assumptions 
\[||\nabla a||^2_{L^2_w(B_1)}=\int_{B_1}w|\nabla a|^2\,dx\,<\infty\quad\text{and}\quad\,||\nabla b||^2_{L^2_{w^{-1}}(B_1)}=\int_{B_1}w^{-1}|\nabla b|^2\,dx<\infty\]
are made and the inequality
\begin{equation*}
    ||\varphi||_{L^\infty(B_1)}\leq C||\nabla a||_{L_{w}^2(B_1)}||\nabla b||_{L_{w^{-1}}^2(B_1)}
\end{equation*}
is then considered.
\vspace{0.5cm}\\
\textbf{Acknowledgments.} The author would like to thank F. Da Lio and T. Rivière for fruitful discussions on the topic as well as for their useful remarks on previous versions of the paper.

\section*{Proof of weighted Wente's inequality}
We prove Theorem \ref{wentetheorem}.
\begin{proof} By Lemma E.2 in \cite{DGR} there exists a dyadic decomposition
\begin{equation}\label{l2convergenceofbj}
    \nabla b=\sum\limits_{j=0}^\infty \nabla b_j
\end{equation}
converging in $L^2$, with 
\begin{equation*}
     spt(\nabla b_j)\subset A_j\cup A_{j+1}=:\tilde{A}_j
\end{equation*}
for the dyadic annuli 
\begin{equation}\label{notationannuli}
    A_j:= B_{2^{-j}}\setminus B_{2^{-j-1}}
\end{equation}
and
\begin{equation}\label{l2boundforbj}
    \int_{B_1}|\nabla b_j|^2\,dx\leq C\int_{\tilde{A}_j}|\nabla b|^2\,dx\,.
\end{equation}
Introducing
    \begin{equation}\label{definitionvarphij}
    \begin{cases}
        \Delta \varphi_j =\nabla a\cdot\nabla^\perp b_j\quad&\text{on }B_1\\[.5em]
        \,\,\,\,\varphi=0 \quad&\text{on }\partial B_1\,,
    \end{cases}
    \end{equation}
we obtain a $L^2$ decomposition of $\varphi$, i.e.
\begin{equation*}
    \varphi=\sum\limits_{j=0}^\infty\varphi_j\,,
\end{equation*}
    where the series converges in $L^2$. Each of the elements in the series has Laplacian in the Hardy space and satisfies Dirichlet boundary conditions. Moreover, the support of the Laplacian of $\varphi_j$ is an annulus with fixed conformal class and radius comparable to $ 2^{-j}$.\\    
   
    \noindent The standard Wente inequality can be applied to each element of the series to obtain
    \begin{equation}\label{wenteforphij}
        \int_{B_1}|\nabla\varphi_j|^2\,dx\leq\,C \int_{B_1}|\nabla b_j|^2\,dx\int_{B_1}|\nabla a|^2\,dx
    \end{equation}
    for a constant $C$ independent of $j$. \\
    In fact the contribution of $\nabla a$ in \eqref{wenteforphij} can be localized: by construction 
    \begin{equation}
         spt(\nabla b_j)\subset \tilde{A}_j=B_{2^{-j}}\setminus B_{2^{-j-2}}
    \end{equation} and choosing a radial cut-off function 
    \begin{equation}
        \psi_j\in[0,1],\quad  spt(\nabla\psi_j)\subset [2^{-j},2^{-j}-2^{-j-3}]\cup[2^{-j-3},2^{-j-2}] ,\quad \psi_j\equiv 1\,\,\text{on}\,\, [2^{-j-2},2^{-j}]\,,
    \end{equation}
    with
    \begin{equation}
        \int_{B_1}|\nabla\psi_j|^2\,dx\leq C\,,
    \end{equation}
    setting $a_j:=\psi_j (a-c_j)$, for a constant $c_j$ to be chosen later,
    \begin{equation}
        \nabla a\cdot\nabla^\perp b_j \overset{L^1}{=}\nabla a_j\cdot\nabla^\perp b_j\,,
    \end{equation}
    \begin{equation}
        \nabla a_j = \psi_j\nabla a + (a-c_j)\nabla \psi_j\,.
    \end{equation}
    Introducing the notation $C_j:= spt( \psi_j)$,
    \begin{equation}
        \int_{B_1}|\nabla a_j|^2\,dx\leq \int_{C_j}|\nabla a|^2\,dx+ C\fint_{C_j}(a-c_j)^2\,dx\leq C\int_{C_j}|\nabla a|^2\,dx\,,
    \end{equation}
    for the appropriate choice of $c_j$ in the Poincaré inequality used in the last step.\\
    In particular,
     \begin{equation}\label{varphijequationbothlocalized}
    \begin{cases}
        \Delta \varphi_j =\nabla a_j\cdot\nabla^\perp b_j\quad&\text{on }B_1\\[.5em]
        \,\,\,\,\varphi=0 \quad&\text{on }\partial B_1\,,
    \end{cases}
    \end{equation}
    and the Wente inequality \eqref{wenteforphij} can therefore be improved to obtain the localized version
     \begin{equation}\label{wenteforphijlocalized}
        \int_{B_1}|\nabla\varphi_j|^2\,dx\leq\,C \int_{B_1}|\nabla b_j|^2\,dx\int_{B_1}|\nabla a_j|^2\,dx\leq\,C \int_{B_1}|\nabla b_j|^2\,dx\int_{C_j}|\nabla a|^2\,dx\,.
    \end{equation}
    We will first prove that the weighted Wente inequality holds for each $\varphi_j$, for $\alpha\in(0,1]$,
    \begin{align}\label{weightedwenteforvarphij}
        \int_{B_1}|x|^{2\alpha}|\nabla\varphi_j|^2\,dx&\leq\,C \int_{B_1}|x|^{2\alpha}|\nabla b_j|^2\,dx\int_{B_1}|\nabla a_j|^2\,dx\\=&\, C \int_{B_1}2^{-2j\alpha}|\nabla b_j|^2\,dx\int_{B_1}|\nabla a_j|^2\,dx 
    \end{align}
    which by summing over $j$ yields
    \begin{equation}
       \sum\limits_{j=0}^\infty  \int_{B_1}|x|^{2\alpha}|\nabla\varphi_j|^2\,dx\leq\,C \int_{B_1}|x|^{2\alpha}|\nabla b|^2\,dx\int_{B_1}|\nabla a|^2\,dx\,.
    \end{equation}
    To obtain the full weighted Wente inequality it then suffices to control the mixed terms
    \begin{equation}
        \sum\limits_{i\neq j}\int_{B_1}|x|^{2\alpha}\nabla\varphi_i\cdot\nabla\varphi_j\,dx\,,
    \end{equation}
    indeed
    \begin{equation}
        \int_{B_1}|x|^{2\alpha}|\nabla\varphi|^2\,dx =\sum\limits_{j=0}^\infty  \int_{B_1}|x|^{2\alpha}|\nabla\varphi_j|^2\,dx+\int_{B_1}\sum\limits_{i\neq j}\int_{B_1}|x|^{2\alpha}\nabla\varphi_i\cdot\nabla\varphi_j\,dx\,.
    \end{equation}
    In order to control the mixed terms, we will make use of the localization both in $b$ and in $a$. We show estimates for functions having Laplacians given by localized Jacobians, which will in particular apply to $\varphi_j$:
    \vspace{0.5cm}\\
Let $u$ be a solution of
\begin{equation}
    \Delta u = \nabla a\cdot\nabla^\perp b_j=div (a\nabla^\perp b_j)\,,
\end{equation}
and $spt(\nabla b_j)\subset B_{2^{-j}}\setminus B_{2^{-j-2}}=\tilde{A}_j$.
Consider
\begin{equation}
    \tilde{u}(x)=\frac{1}{2\pi}\int_{B_1}\log|x-y|\nabla a(y)\cdot\nabla^\perp b_j(y)\,dx\,.
\end{equation}
Then
\begin{equation}
    \Delta\tilde{u}=\nabla a\cdot\nabla^\perp b_j\, \quad in\quad B_1\,.
\end{equation}
For $x\in B_1\setminus B_{2^{-j+1}}$, using the notation $|x|=r$, choosing $c_j:=\fint_{B_{2^{-j}}}a\,dx$,
\begin{align}
    |\tilde{u}|(x)&\leq\int_{A_j}\frac{1}{|x-y|}|(a-c_j)||\nabla b_j|\,dx\\ &\leq C\frac{1}{|x|-2^{-j}}2^{-j}||\nabla a||_2||\nabla b_j||_2\\& = \frac{C}{|x|-2^{-j}}||\nabla a||_2||\,|x|\nabla b_j||_2\,.\label{linfinityestimatesfortildeu}
\end{align}
Multiplying by $|x|^\alpha$, for $x\in B_1\setminus B_{2^{-j+1}}$,
\begin{align}
    |x|^\alpha|\tilde{u}|&\leq \frac{C}{|x|-2^{-j}}|x|^\alpha2^{-j}||\nabla a||_2||\nabla b_j||_2\leq C\frac{|x|^\alpha2^{-j(1-\alpha)}}{|x|-2^{-j}}||\nabla a||_2||\,|x|^\alpha\nabla b_j||_2\\
    &\leq C\frac{r}{|x|-2^{-j}}\bigg(\frac{2^{-j}}{r}\bigg)^{(1-\alpha)}||\nabla a||_2||\,|x|^\alpha\nabla b_j||_2\,.\label{estimateintermediatefortildeujsup}
\end{align}
Applying the maximum principle for harmonic functions on $B_1$ to $v=\tilde{u}-u$ we see that
\begin{equation}\label{maximumprinciplebounds}
    |u|\leq |\tilde{u}|+C(B_1)||\nabla a||_2||\,|x|\nabla b_j||_2\,,
\end{equation}
so that for any $r\geq 2^{-j+1}$, by \eqref{estimateintermediatefortildeujsup} and \eqref{maximumprinciplebounds},
\begin{equation}\label{finalestimateweightedlinfinity}
    |x|^{2\alpha}| u|^2\leq C|x|^{2\alpha}|\tilde{u}|^2+C|x|^{2\alpha}||\nabla a||_2^2||\,|x|\nabla b_j||_2^2\leq C||\nabla a||_2^2||\,|x|^\alpha\nabla b_j||_2^2\,.
\end{equation}
On the other hand for $r\leq 2^{-j+1}$, from the standard Wente inequality we can conclude
\begin{equation}
    |x|^{2\alpha}|u|^2\leq C2^{-2j\alpha}||\nabla a||_2^2||\nabla b_j||_2^2.
\end{equation}
We have therefore shown
\begin{Lm}
    Let $u\in W^{1,2}_0(B_1)$ be the unique solution of 
     \begin{equation}
    \begin{cases}
        \Delta u =\nabla a\cdot\nabla^\perp b_j\quad&\text{on }B_1\\[.5em]
        \,\,\,\,u=0 \quad&\text{on }\partial B_1\,,
    \end{cases}
    \end{equation}
    with $spt(\nabla b_j)\subset\tilde{A}_j$. Then, for any $0\leq \alpha\leq 1$,
    \begin{equation}
        ||\,|x|^\alpha u||_{L^\infty(B_1)}\leq C||\nabla a||_2||\,|x|^\alpha\nabla b_j||_2\,.
    \end{equation}
\end{Lm}
\noindent Applying \eqref{finalestimateweightedlinfinity} to $\varphi_j$, which solves \eqref{varphijequationbothlocalized}, yields
\begin{equation}
        ||\,|x|^\alpha \varphi_j||_{L^\infty(B_1)}\leq C||\nabla a_j||_2||\,|x|^\alpha\nabla b_j||_2\,.
    \end{equation}
and by taking the sum we obtain,
\begin{align}
    |x|^{2\alpha}|\varphi|^2&\leq \sum\limits_{i,j=0}^\infty |x|^{2\alpha}|\varphi_i||\varphi_j|\leq C\sum\limits_{i,j=0}^\infty||\nabla a_i||_2||\,|x|^\alpha\nabla b_i||_2||\nabla a_j||_2||\,|x|^\alpha\nabla b_j||_2\\&\leq C\sum\limits_{i,j=0}^\infty||\nabla a_i||_2^2||\,|x|^\alpha\nabla b_j||_2^2+||\nabla a_j||_2^2||\,|x|^\alpha\nabla b_i||_2^2\leq C||\nabla a||_2^2||\,|x|^\alpha\nabla b||_2^2\,.
\end{align}
To summarize we proved the first part of the main Theorem \ref{wentetheorem},
    \begin{Lm}
        Let $\varphi\in W^{1,2}_0(B_1)$ solve
         \begin{equation*}
    \begin{cases}
        \Delta \varphi =\nabla a\cdot\nabla^\perp b\quad&\text{on }B_1\\[.5em]
        \,\,\,\,\varphi=0 \quad&\text{on }\partial B_1\,.
    \end{cases}
    \end{equation*}
     Then for all $0\leq\alpha\leq 1$,
         \begin{equation*}
            ||\,|x|^\alpha \varphi||_{L^\infty(B_1)}\leq C||\nabla a ||_{L^2(B_1)}||\,|x|^\alpha\nabla b ||_{L^2(B_1)}\,.
        \end{equation*}
    \end{Lm}
\noindent We now move on to the proof of the $W^{1,2}$ weighted Wente estimates. We can obtain pointwise estimates for the gradient of $\tilde{u}$ in the region $x\in B_1\setminus B_{2^{-j+1}}$ directly from the definition of $\tilde{u}$,
\begin{align}
    |\nabla\tilde{u}|(x)&\leq\int_{B_{2^{-j}}}\frac{1}{|x-y|^2}|(a-c_j)||\nabla b_j|\,dx\\ &\leq C\frac{1}{(|x|-2^{-j})^2}2^{-j}||\nabla a||_2||\nabla b_j||_2\\& = \frac{C}{(|x|-2^{-j})^2}||\nabla a||_2||\,|x|\nabla b_j||_2\,,
\end{align}
which multiplied by the weight $|x|^\alpha$, gives for $r>2^{-j+1}$
\begin{align}
    |x|^\alpha|\nabla \tilde{u}|(x)&\leq  C\frac{|x|^\alpha 2^{-j(1-\alpha)}}{(|x|-2^{-j})^2}||\nabla a||_2||\,|x|^\alpha\nabla b_j||_2\\ &\leq C\frac{r}{(|x|-2^{-j})^2}\bigg(\frac{2^{-j}}{r}\bigg)^{1-\alpha}||\nabla a||_2||\,|x|^\alpha \nabla b_j||_2\\&\leq C\frac{1}{r}\bigg(\frac{2^{-j}}{r}\bigg)^{1-\alpha}||\nabla a||_2||\,|x|^\alpha \nabla b_j||_2\,.
\end{align}
For the annulus $A_k:=B_{2^{-k}}\setminus B_{2^{-k-1}}$, with $k<j-1$,
\begin{equation}
 \int_{A_k}|x|^{2\alpha}|\nabla\tilde{u}|^2\,dx\leq C2^{-2(j-k)(1-\alpha)}||\nabla a_j||_2^2||\,|x|^\alpha\nabla b_j||_2^2\,,
\end{equation}
and summing over $k$,
\begin{align}
    &\int_{B_1\setminus B_{2^{-j+1}}}|x|^{2\alpha}|\nabla\tilde{u}|^2 \leq \,C||\nabla a||_2^2||\,|x|^\alpha\nabla b_j||_2^2\sum\limits_{k=0}^{j-2}2^{-2(1-\alpha)(j-k)}\\\leq&\,\begin{cases}
        C||\nabla a||_2^2||\,|x|^\alpha\nabla b_j||_2^2\quad \text{for}\quad 0<\alpha<1\\[.5em]Cj||\nabla a||_2^2||\,|x|\nabla b_j||_2^2\leq C||\nabla a||_2^2||r\sqrt{|\log(r)|}\nabla b_j||_2^2\quad \text{for}\quad \alpha=1\,.\label{partialgradientestimateinl2}
    \end{cases}
\end{align}
Recall that
\begin{equation}
    ||\tilde{u}||_{L^\infty(\partial B_1)}\leq C||\nabla a||_2||\,|x|\nabla b_j||_2 \,,
\end{equation}
and hence for the harmonic difference $\tilde{u}-u$,
\begin{equation}
    ||\nabla (\tilde{u}-u)||_{L^\infty(B_\eta)}\leq C_\eta||\nabla a||_2||\,|x|\nabla b_j||_2\,.
\end{equation}
In particular,
\begin{equation}
    \int_{B_{1/2}\setminus B_{2^{-j+1}}}|x|^{2\alpha}|\nabla u|^2\,dx\leq \int_{B_{1/2}\setminus B_{2^{-j+1}}}|x|^{2\alpha}|\nabla \tilde{u}|^2\,dx+C||\nabla a||_2^2||\,|x|\nabla b_j||_2^2\,.
\end{equation}
Moreover,
\begin{equation}
    \sup\limits_{B_1}|\tilde{u}-u|\leq C||\nabla a||_2||\,|x|\nabla b_j||_2\,,
\end{equation}
so that
\begin{equation}
    \int_{B_1\setminus B_{1/2}}|x|^{2\alpha}|\nabla u|^2\,dx\leq \int_{\partial B_{1/2}}|u||\nabla u|\,dx\leq C||\nabla a||_2^2||\,|x|\nabla b_j||_2^2\,.
\end{equation}
It therefore suffices to prove for some $f:(0,\frac{1}{2})\to\mathbb{R}_+$, with $|x|^\alpha\leq Cf(r)$ that
\begin{equation}\label{eqneededwithweightrhs}
    \int_{B_1\setminus B_{2^{-j+1}}}|x|^{2\alpha}|\nabla\tilde{u}|^2\,dx\leq C||\nabla a||_2^2||f(r)\nabla b_j||_2^2
\end{equation}
to obtain
\begin{equation}
    \int_{B_1\setminus B_{2^{-j+1}}}|x|^{2\alpha}|\nabla u|^2\,dx\leq C||\nabla a||_2^2||f(r)\nabla b_j||_2^2\,,
\end{equation}
which together with 
\begin{equation}
    \int_{ B_{2^{-j+1}}}|x|^{2\alpha}|\nabla u|^2\,dx\leq C2^{-2j\alpha}||\nabla a||_2^2||\nabla b_j||_2^2\leq C||\nabla a||_2^2||\,|x|^\alpha\nabla b_j||_2^2\,,
\end{equation}
which follows from the standard Wente inequality would imply
\begin{equation}
    \int_{B_1}|x|^{2\alpha}|\nabla u|^2\,dx\leq C||\nabla a||_2^2||f(r)\nabla b_j||_2^2\,.
\end{equation}
In \eqref{partialgradientestimateinl2} we have shown that \eqref{eqneededwithweightrhs} hold for $f(r)=r\sqrt{|\log(r)|}$ in the case $\alpha =1$ and $f(r)=|x|^\alpha$ in the case $0<\alpha<1$. We conclude:
\begin{Lm}\label{intermediatelemma2}
    Let $u\in W^{1,2}_0(B_1)$ be the solution of the localized Wente equation
     \begin{equation*}
    \begin{cases}
        \Delta u =\nabla a\cdot\nabla^\perp b_j\quad&\text{on }B_1\\[.5em]
        \,\,\,\,u=0 \quad&\text{on }\partial B_1\,,
    \end{cases}
    \end{equation*}
    $a,b_j\in W^{1,2}(B_1)$ and $spt(\nabla b_j)\subset\tilde{A_j}$. Then for all $0<\alpha<1$ there exists a constant $C_\alpha>0$ such that
    \begin{equation*}
        \int_{B_1}|x|^{2\alpha}|\nabla u|^2\,dx\leq C_\alpha \int_{B_1}|\nabla a|^2\,dx\int_{B_1}|x|^{2\alpha}|\nabla b_j|^2\,dx.
    \end{equation*}
    Moreover there exists a universal constant $C>0$, such that
     \begin{equation*}
        \int_{B_1}|x|^{2}|\nabla u|^2\,dx\leq C \int_{B_1}|\nabla a|^2\,dx\int_{B_1}|x|^{2}|\log|x||\,|\nabla b_j|^2\,dx.
    \end{equation*}
\end{Lm}
\noindent We can now conclude the proof of Theorem \ref{wentetheorem} by applying Lemma \ref{intermediatelemma2} to each $\varphi_j$:
    \begin{Lm}
        Let $\varphi\in W^{1,2}_0(B_1)$ solve
         \begin{equation*}
    \begin{cases}
        \Delta \varphi =\nabla a\cdot\nabla^\perp b\quad&\text{on }B_1\\[.5em]
        \,\,\,\,\varphi=0 \quad&\text{on }\partial B_1\,,
    \end{cases}
    \end{equation*}
    Then for all $0<\alpha<1$ there exists $C_\alpha>0$, such that
         \begin{equation*}
            ||\,|x|^\alpha\nabla\varphi||_{L^2(B_1)}\leq C||\nabla a||_{L^2(B_1)}||\,|x|^\alpha\nabla b||_{L^2(B_1)}\,.
        \end{equation*}
        Moreover there exists a universal constant $C>0$, such that
        \begin{equation}\label{weightedwentewithextralog}
            ||\,|x|\nabla\varphi||_{L^2(B_1)}\leq C||\nabla a||_{L^2(B_1)}||r\sqrt{|\log|x||}\nabla b||_{L^2(B_1)}\,.
        \end{equation}       
    \end{Lm}
\begin{proof}
    By the decomposition introduced above, and Lemma \ref{intermediatelemma2} applied to each $\varphi_i$,
    \begin{align}
||\,|x|^\alpha\nabla\varphi||_{L^2}^2\leq &C\sum\limits_{i,j=0}^\infty||\,|x|^\alpha\nabla\varphi_i||_{L^2}||\,|x|^\alpha\nabla\varphi_j||_{L^2}\\\leq &\begin{cases}
    C\sum\limits_{i,j=0}^\infty||\nabla a_i||_2||\,|x|^\alpha\nabla b_i||_2||\nabla a_j||_2||\,|x|^\alpha\nabla b_j||_2\quad\text{for}\,0<\alpha<1\\[.5em]
    C\sum\limits_{i,j=0}^\infty||\nabla a_i||_2||r\sqrt{\log(r)}\nabla b_i||_2||\nabla a_j||_2||r\sqrt{\log(r)}\nabla b_j||_2\quad\text{for}\, \,\,\alpha=1
\end{cases}\nonumber\\
\leq &\begin{cases}
    C_\alpha||\nabla a||_2^2||\,|x|^\alpha \nabla b||_2^2\quad\text{for}\,0<\alpha<1\\[.5em]
      C||\nabla a||_2^2||r\sqrt{\log(r)} \nabla b||_2^2\quad\text{for}\, \,\,\alpha=1\,.
\end{cases}
    \end{align}
    This concludes the proof of the Lemma.
\end{proof}
\noindent We concluded the first part of the paper, namely the proof of Theorem \ref{wentetheorem}.
\end{proof}
\section*{The case $\alpha=1$}
The goal of this section is to study the case $\alpha=1$, in particular showing in Theorem 2 that the weight $|x|^2\log|x|$ obtained in Theorem \ref{wentetheorem} is optimal.\\[.5em]
We start by making the following observation: assume $\varphi\in C^0(B_1)\cap W^{1,2}_0(B_1)$, $\Delta \varphi=0$ in $B_1\setminus B_s$ and $\Delta\varphi=\nabla a\cdot\nabla^\perp b$ in $B_s$ as well as $\partial_r\varphi=\partial_{-r}\varphi$ on $\partial B_s$, then 
\begin{equation}
    \Delta\varphi=\nabla a\cdot\nabla^\perp b\quad \text{in}\,\,\, B_1\,.
\end{equation}
Indeed, for $\psi\in W^{1,2}_0(B_1)$,
\begin{align}
    \int_{B_1}\nabla \varphi \nabla\psi\,dx&=\int_{\partial(B_1\setminus B_s)}\partial_\nu\varphi\psi\,dx-\int_{B_s}\nabla a\cdot\nabla^\perp b \psi\,dx+\int_{\partial B_s}\partial_\nu\varphi\psi\,dx\\
    &=-\int_{B_s}\nabla a\cdot\nabla^\perp b\psi\,dx+\int_{\partial B_s}(\partial_r\varphi-\partial_{-r}\varphi)\psi\,dx\\
    &=-\int_{B_s}\nabla a\cdot\nabla^\perp b\psi\,dx\,.
\end{align}
Consider now the following harmonic function on $B_1\setminus \{0\}$:
\begin{equation}
    h(x)=\frac{x}{|x|^2}-x=\partial_x\log(|x|)-x\,.
\end{equation}
Notice that this function satisfies
\begin{equation}
    \partial_rh(x)=-\frac{x(1+|x|^2)}{|x|^3},
\end{equation}
and hence $\int_{\partial B_s}\partial_rh\,dx=0$ for all $0<s<1$, which is a necessary condition in order to be a solution to a Wente type equation. The gradient of $h$ is given by $|\nabla h|^2=|x|^{-4}+1$, and hence
\begin{equation}\label{growthofunweightedharmonicpart}
    \int_{B_1\setminus B_s}|\nabla h|^2\,dx\leq Cs^{-2}\,,
\end{equation}
for all $\beta\in(0,1)$,
\begin{equation}\label{growthofweightedharmonicpartbeta}
    \int_{B_1\setminus B_s}|x|^{2\beta}|\nabla h|^2\,dx\leq Cs^{-2+2\beta}\,,
\end{equation}
and
\begin{equation}\label{growthofcriticallyweightedharmonicpart}
    \int_{B_1\setminus B_{s}}|x|^2|\nabla h|^2\,dx\geq\int_{B_1\setminus B_{s}}\frac{x^2}{|x|^4}\,dx=\,C\int_0^{2\pi}\cos^2(\varphi)\,d\varphi\int_{s}^{1}\frac{1}{r}dr=\,C |\log(s)|\,.
\end{equation}
We produce a counterexample to the weights $|x|^2$ and $|x|^2|\log|x||^{1-\epsilon}$ in the weighted Wente inequality, by gluing a function $f$ on $B_{s}$ solving a Wente equation to $h$.\\[.5em]
Let $f=x|x|^\alpha$. Then
\begin{equation}
    \partial_rf=\frac{x}{|x|}|x|^\alpha(1+\alpha)\,,
\end{equation}
and
\begin{equation}
    \Delta f = 2\partial_x|x|^\alpha+x\Delta |x|^\alpha=x|x|^{\alpha-2}\alpha(\alpha+2)=\nabla a_\alpha\cdot\nabla^\perp b\,,
\end{equation}
for $b=y$ and $a_\alpha=(\alpha+2)|x|^\alpha$, which gives $\partial_x a_\alpha= \alpha x|x|^{\alpha-2}(\alpha+2)$.\\
In order to glue $h$ and $K(s,\alpha)f$ continuously along $\partial B_s$ the following needs to be satisfied:
\begin{equation}
    \frac{1}{s^2}-1=K(s,\alpha)s^{\alpha}\implies K(s,\alpha)=s^{-\alpha-2}(1-s^2)\,,
\end{equation}
the normal derivatives then satisfy on $\partial B_s$,
\begin{equation}
    \partial_rK(s,\alpha)f=xs^{\alpha-1}(1+\alpha)s^{-\alpha-2}(1-s^2)=x(1+\alpha)s^{-3}(1-s^2)\,.
\end{equation}
On the other hand, still on $\partial B_s$,
\begin{equation}
    \partial_{-r }h= xs^{-3}{(1+s^2)}
\end{equation}
and in order for the gluing to satisfy $\partial_r=\partial_{-r}$ on $\partial B_s$ we need,
\begin{equation}
    1+s^2=(1+\alpha)(1-s^2)\,,
\end{equation}
which is satisfied for $s_\alpha^2=\frac{\alpha}{2+\alpha}$. We define
\begin{equation}
    h_\alpha=\begin{cases}
        h\quad  &\text{on} \,\,B_{1}\setminus B_{s_\alpha}\\[.5em]
        K(s_\alpha,\alpha)f\quad &\text{on}\,\, B_{s_\alpha}\,.
    \end{cases}
\end{equation}
Then, by the previous discussion
\begin{equation}
    \Delta h_\alpha=\nabla \tilde{a}_\alpha\cdot\nabla^\perp \tilde{b}_\alpha
\end{equation}
for $\tilde{b}=y$ and $\tilde{a}_\alpha=K(s_\alpha,\alpha)(\alpha+2)|x|^\alpha$ on $B_{s_\alpha}$ and $\tilde{a}_\alpha=K(s_\alpha,\alpha)(\alpha+2)s_\alpha^\alpha$ on $B_1\setminus B_{s_\alpha}$. A direct computation shows:

\begin{align}  &\int_{B_1}|\nabla\tilde{a}_\alpha|^2\,dx=\int_{B_{s_\alpha}}K(s_\alpha,\alpha)^2|\nabla a_\alpha|^2\,dx\\=&\int_{B_{s_\alpha}}s_\alpha^{-2\alpha-4}(1-s_\alpha^2)^2\alpha^2|x|^{2\alpha-2}(\alpha+2)^2\,dx\\
=&\,Cs_\alpha^{-2\alpha-4}(1-s_\alpha^2)^2\alpha^2(\alpha+2)^2\int_0^{s_\alpha}|x|^{2\alpha-1}\,dx\\
=&\frac{Cs_\alpha^{-2\alpha-4}(1-s_\alpha^2)^2\alpha^2(\alpha+2)^2}{2\alpha}s_\alpha^{2\alpha}\\
=&\,C\alpha(\alpha+2)^2 s_\alpha^{-4}(1-s_\alpha^2)^2\\
=&\,C\alpha \frac{(2+\alpha)^2}{\alpha^2}\frac{4(2+\alpha)^2}{(2+\alpha)^2} \\ =&\,C\frac{(2+\alpha)^2}{\alpha}=\,C\frac{2+\alpha}{s_\alpha^2}\to\infty \quad\text{as}\,\,\alpha\to 0\,,
\end{align}
which is expected given the standard Wente estimate and the growth rate we computed in \eqref{growthofunweightedharmonicpart}. Similarly from the computations above we deduce
\begin{equation}
    \int_{B_{s_\alpha}\setminus B_{s_\alpha/2}}|\nabla\tilde{a}_\alpha|^2\,dx=\,C\,\frac{2+\alpha}{s_\alpha^2}\,,
\end{equation}
and hence for $\beta\in(0,1)$
\begin{equation}
    \int_{B_1}|x|^{2\beta}|\nabla\tilde{a}_\alpha|^2\,dx\geq C(2+\alpha)s_\alpha^{2\beta-2}\,,
\end{equation}
which is consistent with \eqref{growthofweightedharmonicpartbeta} and the main Theorem \ref{wentetheorem}.
On the other hand,
\begin{equation}
    \int_{B_1}|x|^{2}|\nabla\tilde{a}_\alpha|^2\,dx\leq C,
\end{equation}
which shows that the case $\alpha=1$ with the weight $|x|^2$ cannot hold. In fact more can be said, for any $\beta\in (0,1)$ and $s_\alpha$ small enough,
\begin{equation}
    \int_{B_1}|x|^2|\log(|x|)|^\beta|\nabla\tilde{a}_\alpha|^2\,dx\leq Cs_\alpha^2|\log(s_\alpha)|^\beta\int_{B_1}|\nabla\tilde{a}_\alpha|^2\,dx\leq C|\log(s_\alpha)|^\beta
\end{equation}
and hence a Weighted Wente inequality of the form $||\,|x|\nabla\varphi||_2\leq||\nabla b||_2||r|\log(r)|^{\beta/2}\nabla a||_2$ cannot hold for any $\beta <1$. This shows that our main Theorem \ref{wentetheorem} is optimal, and
\begin{equation*}
        \forall \beta<1:\quad\sup\limits_{\varphi\in\mathcal{A}}\frac{||\,|x|\nabla\varphi||_2}{||\nabla a||_2||\,|x|\,|\log|x||^{\beta/2}\nabla b||_2}=\infty\,,
    \end{equation*}
which also concludes the proof of Theorem 2.


\newpage

  \begin{thebibliography}{99}

\bibitem{Bar} Baraket, S. Estimations of the best constant involving the $L^ \infty $ norm in Wente’s inequality. Annales de la Faculté des sciences de Toulouse : Mathématiques, Serie 6, Volume 5 (1996) no. 3, pp. 373-385. 

\bibitem{BC} Baraket, S and Chaabane, L. (2003). The Wente inequality on weighted Sobolev spaces. Houston Journal of Mathematics. 29. 

\bibitem{BR} Bernard, Y., and Rivière T., “Energy Quantization for Willmore Surfaces and Applications.” Annals of Mathematics 180, no. 1 (2014): 87–136. http://www.jstor.org/stable/24522919.

\bibitem{Bet} Bethuel, F. ”Un résultat de régularité pour les solutions de l’équation de surfaces courbure moyenne prescrite.” (French) [A regularity result for solutions to the equation of surfaces of prescribed mean curvature] C. R. Acad. Sci. Paris Sr. I Math. 314 (1992), no. 13, 1003–1007.

\bibitem{BG} Bethuel, F.; Ghidaglia, J.-M. Improved regularity of solutions to elliptic equations involving Jacobians and applications. J. Math. Pures Appl. (9) 72 (1993), no. 5, 441–474.

\bibitem{CLMS} Coifman, R.; Lions, P.-L.; Meyer, Y.; Semmes, S. ”Compensated compactness and Hardy spaces”. J. Math. Pures Appl. (9) 72 (1993), no. 3, 247–286.

\bibitem{DLP} Da Lio, F. and Palmurella, F. (2017) Remarks on Neumann boundary problems involving Jacobians, Communications in Partial Differential Equations, 42:10, 1497-1509, DOI: 10.1080/03605302.2017.1377231

\bibitem{DGR} Da Lio, F., Gianocca, M., Rivière, T. Morse Index Stability for Critical Points to Conformally invariant Lagrangians, preprint Arxiv arXiv:2212.03124


\bibitem{Ge} Ge, Y. (1998). Estimations of the best constant involving the $L^2$ norm in Wente's inequality and compact H-surfaces in Euclidean space. ESAIM: Control, Optimisation and Calculus of Variations, 3, 263-300. doi:10.1051/cocv:1998110

\bibitem{Ge2} Ge, Y. 1999. A remark on generalized harmonic maps into spheres. Nonlinear Anal. 36, 4 (May 1999), 495–506. https://doi.org/10.1016/S0362-546X(98)00079-0


\bibitem{Hel2} Hélein, F. “Régularité des applications faiblement harmoniques entre une surface et une sphère.” (1990).  

\bibitem{Hel} Hélein, F. (2002). Harmonic Maps, Conservation Laws and Moving Frames (2nd ed., Cambridge Tracts in Mathematics). Cambridge: Cambridge University Press. doi:10.1017/CBO9780511543036

\bibitem{Hir} Hirsch, J. (2019) Nonexistence of Wente’s $L^\infty$ estimate for the Neumann problem. Analysis and PDE 12:4, pages 1049-1063.

\bibitem{LR} Laurain, P. and Rivière, T. "Angular energy quantization for linear elliptic systems with antisymmetric potentials and applications." Anal. PDE 7 (1) 1 - 41, 2014. https://doi.org/10.2140/apde.2014.7.1

\bibitem{LR2} Lin F.-H. and Rivière T. Energy quantization for harmonic maps. Duke Math. J. 111:1
(2002), 177-193.

\bibitem{LR1} Lin, F.-H. and Rivière, T. A quantization property for moving line vortices. Comm. Pure Appl. Math. 54 (2001),

 \bibitem{Par} Parker T.. "Bubble tree convergence for harmonic maps." Journal of Differential Geometry, 44(3)
595-633 1996.

\bibitem{Riv}  Rivi\`ere T. {\em Conservation laws for conformally invariant variational problems.} Invent. Math. 168 (2007), no. 1, 1--22.

\bibitem{Sch}Schikorra, A. Boundary Equations and Regularity Theory for Geometric Variational Systems with Neumann Data. Arch Rational Mech Anal 229, 709–788 (2018). https://doi.org/10.1007/s00205-018-1226-4

 \bibitem{Str}Strzelecki, P. Regularity of p-harmonic maps from the p-dimensional ball into a sphere. Manuscripta Math 82, 407–415 (1994). https://doi.org/10.1007/BF02567710

 \bibitem{Top} Topping P., The optimal constant in Wente's $L^\infty$ estimate. Comment. Math. Helv. 72 (1997), no. 2, pp. 316–328

\bibitem{Wen}Wente, H. "A General Existence Theorem for Surfaces of Constant Mean Curvature.." Mathematische Zeitschrift 120 (1971): 277-288.

\bibitem{Wen2} Wente, H. Large solutions to the volume constrained Plateau problem. Arch. Rational Mech. Anal., 75(1):59–77, 1980/81.


\end{thebibliography}
\end{document}